\documentclass[12pt,notitlepage]{amsart}
\usepackage{latexsym,amsfonts,amssymb,amsmath,amsthm} 
\newcommand{\mz}{\ensuremath{\mathbb Z}}
\newcommand{\mr}{\ensuremath{\mathbb R}}

\newcommand{\mq}{\ensuremath{\mathbb Q}}

\newcommand{\mymod}{\ensuremath{\negthickspace \negmedspace \pmod}}
\newcommand{\shortmod}{\ensuremath{\negthickspace \negthickspace \negthickspace \pmod}}

\newcommand{\half}{\ensuremath{ \frac{1}{2}}}

\newcommand{\intR}{\int_{-\infty}^{\infty}}
\newcommand{\sumstar}{\sideset{}{^*}\sum}
\newcommand{\sumflat}{\sideset{}{^\flat}\sum}

\newcommand{\sumprime}{\sideset{}{'}\sum}

\newcommand{\leg}[2]{\left(\frac{#1}{#2}\right)}

\newcommand{\muK}{\mu_{\omega}}

\theoremstyle{plain}		
	\newtheorem{mytheo}{Theorem}[section]

	\newtheorem{myprop}[mytheo]{Proposition}
	\newtheorem{mycoro}[mytheo]{Corollary}
     \newtheorem{mylemma}[mytheo]{Lemma}

\theoremstyle{remark}

\begin{document}
\today
\title{Mean values with cubic characters}
\author{Stephan Baier}
\address{Jacobs University \\
School of Engineering and Science \\
P.O. Box 750561 \\
28725 Bremen \\
Germany }
 \email{s.baier@jacobs-university.de}

\author{Matthew P. Young}
\address{Department of Mathematics \\
	  Texas A\&M University \\
	  College Station \\
	  TX 77843-3368 \\
		U.S.A.}
\email{myoung@math.tamu.edu}

\begin{abstract}
We investigate various mean value problems involving order three primitive Dirichlet characters.  In particular, we obtain an asymptotic formula for the first moment of central values of the Dirichlet $L$-functions associated to this family, with a power saving in the error term.  We also obtain a large-sieve type result for order three (and six) Dirichlet characters.
\end{abstract}

\maketitle

\section{Introduction and Main results}
Dirichlet characters of a given order appear naturally in many applications in number theory.  The quadratic characters have seen a lot of attention due to attractive questions to ranks of elliptic curves, class numbers, etc., yet the cubic characters have been relatively neglected. In this article we are interested in mean values of $L$-functions twisted by characters of order three, and also large sieve-type inequalities for these characters.

Our first result on such $L$-functions is the following
\begin{mytheo}
\label{thm:mainLvalueresult}
Let $w: (0,\infty) \rightarrow \mr$ be a smooth, compactly supported function.  Then
\begin{equation}
\label{eq:1}
 \sum_{(q,3)=1}\;  \sumstar_{\substack{\chi \shortmod{q} \\ \chi^3 = \chi_0}} L(\tfrac12, \chi) w\leg{q}{Q} = c Q \widehat{w}(0) + O(Q^{37/38 + \varepsilon}),
\end{equation}
where $c > 0$ is a constant that can be given explicitly in terms of an Euler product (see \eqref{eq:c} below), and $\widehat{w}$ is the Fourier transform of $w$.  Here the $*$ on the sum over $\chi$ restricts the sum to primitive characters, and $\chi_0$ denotes the principal character.
\end{mytheo}

This result is most similar (in terms of method of proof) to the main result of \cite{Luo}, who considered the analogous mean value but for the case of cubic Hecke $L$-functions on $\mq(\omega)$, $\omega = e^{2\pi i/3}$.  
Our problem has new analytic difficulties which we briefly sketch here.  It turns out that the sum over cubic characters can be parameterized roughly as characters of the form $\chi_n(m) = \leg{m}{n}_3$, the cubic residue symbol, where $n$ runs over elements of $\mz[\omega]$ (see Lemma \ref{lemma:cubicclass} for a precise statement).  Applying an approximate functional equation and reversing the orders of summation leads to the problem of estimating sums of the form
\begin{equation*}
\label{eq:2}
 M_1 =\sum_{m \leq A} \frac{1}{\sqrt{m}} \sum_{N(n) \leq Q} \leg{m}{n}_3, \quad \text{and} \quad M_2=\sum_{m \leq B} \frac{1}{\sqrt{m}} \sum_{N(n) \leq Q} \frac{\tau(\chi_n)}{\sqrt{N(n)}} \leg{m}{n}_3,
\end{equation*}
where $\tau(\chi_n)$ is the cubic Gauss sum and $AB = Q$.  The analogous quantities considered by Luo are similar except the sum over $m$ instead runs over elements of $\mz[\omega]$ with $N(m) \leq A, B$ respectively, again with $AB=Q$.  It is perhaps most natural to view the rational integers $m \leq A$ as a thin subset of the elements of $\mz[\omega]$ with norm $\leq A^2$.  This is natural because many  transformations or estimates have quality related to the norm of $m$. For example, the analog of P\'{o}lya-Vinogradov says that if $m$ is not a cube then $S=\sum_{N(n) \leq Q} \leg{m}{n}_3$, or at least a smoothed version of $S$, is $\ll N(m)^{\half + \varepsilon} = m^{1 + 2\varepsilon}$.  

As an aside, it should not be surprising that $c > 0$ in \eqref{eq:1} since the set of central values is invariant under complex conjugation.

Other authors (\cite{FHL}, \cite{FrHL}, \cite{Diaconu},  \cite{BFH}, \dots) have considered cubic and higher order twists using multiple Dirichlet series.  However, the method using the metaplectic Eisenstein series currently requires the ground field to contain the $l$-th roots of unity (supposing one is twisting by order $l$ Hecke characters).  
Diaconu and Tian \cite{DiTi} have developed analytic properties of a multiple Dirichlet series that potentially has applications to the first moment considered in our Theorem 1.1.  By taking $r=3$, $F=\mathbb{Q}$, and $L=\mathbb{Q}(\omega)$ (in their Section 3) they obtain a double Dirichlet series roughly of the form
\begin{equation*}
 \sum_{m \in \mathbb{N}} \sum_{n \in \mz[\omega]} \frac{\leg{m}{n}_3}{m^s N(n)^w}, \qquad \text{Re}(s), \text{Re}(w) > 1.
\end{equation*}
The meromorphic continuation, location of poles, and order of growth of this double Dirichlet series allows one to consider moments similar to that of Theorem \ref{thm:mainLvalueresult}.  We thank an anonymous referee for pointing this out to us.  However, it is not clear if our Theorem \ref{thm:mainLvalueresult} can be obtained from \cite{DiTi}.  

A consequence of Theorem \ref{thm:mainLvalueresult} is

\begin{mycoro}
\label{coro:nonvanish}
 There exist infinitely many primitive Dirichlet characters $\chi$ of order $3$ such that $L(1/2, \chi) \neq 0$.  More precisely, the number of such characters with conductor $\leq Q$ 
is $\gg Q^{6/7 - \varepsilon}$.
\end{mycoro}

\begin{proof}
Let $N_3(Q)$ be the number of primitive Dirichlet characters of order 3 with conductor $\le Q$ such that $L(1/2,\chi)$ does not vanish. Then using Theorem \ref{thm:mainLvalueresult}, H\"older's inequality and the familiar eight moment bound 
\begin{equation} \label{eightmomentbound}
\sum_{q\le Q}\;  \sumstar_{\chi \shortmod{q}}|L(1/2, \chi)|^8\ll Q^{2+\varepsilon}
\end{equation}
for the family of all primitive Dirichlet characters with conductor $\le Q$
(see Theorem 7.34 of \cite{IK}), we obtain
\begin{equation*}
Q \ll \sum_{q\le Q}\;  \sumstar_{\substack{\chi \shortmod{q} \\ \chi^3 = \chi_0}}
|L(1/2, \chi)| \ll  \big(\sum_{q\le Q}\;  \sumstar_{\chi \shortmod{q}}|L(1/2, \chi)|^8\big)^{1/8} N_3(Q)^{7/8} \ll Q^{1/4+\varepsilon}  N_3(Q)^{7/8}
\end{equation*}
which gives $Q^{6/7-\varepsilon}\ll N_3(Q)$.  \end{proof}

The work of \cite{DiTi} gives a non-quantitative version of Corollary \ref{coro:nonvanish}.
Studies of the moments and nonvanishing of cubic twists of elliptic curves using random matrix theory have been carried out in \cite{David}.  Nonvanishing of cubic twists of elliptic curves using algebraic methods has been undertaken by \cite{FKK}.

As for the second moment, we show
\begin{mytheo} \label{secondmoment}
Let $Q\ge 1$. Then we have
\begin{equation} \label{secondm}
\sum\limits_{\substack{q\le Q}}\ \sumstar\limits_{\substack{\chi \shortmod q\\ \chi^3=\chi_0}} \left| L(1/2+it,\chi) \right|^2 \ll Q^{6/5+\varepsilon} (1+|t|)^{6/5+\varepsilon}.
\end{equation}
For a rational integer $m$, let 
\begin{equation}
\label{eq:HeckeL}
 L(s, \psi_m) =\sum_{\substack{n \in \mz[\omega] \\ n \equiv 1 \shortmod{3}}} \leg{m}{n}_3 N(n)^{-s}
\end{equation}
denote the Hecke L-function.  Then
\begin{equation}
\label{eq:secondmomentHecke}
 \sumstar_{m \leq M} |L(1/2 + it, \psi_m)|^2 \ll M^{3/2 + \varepsilon} (1+|t|)^{4/3},
\end{equation}
where the star indicates the sum is over squarefree integers.
\end{mytheo}
Our proof of Theorem \ref{thm:mainLvalueresult} uses \eqref{eq:secondmomentHecke} as a key ingredient; actually, we require a minor variant given by \eqref{eq:Heckevariant} below.
We shall establish Theorem \ref{secondmoment} by using the following large sieve-type result with cubic Dirichlet characters.

\begin{mytheo} \label{cubiclargesieve}
Let $(a_m)_{m\in \mathbb{N}}$ be an arbitrary sequence of complex numbers. Then
\begin{equation} \label{final}
\sum_{Q<q\le 2Q} \; \sumstar_{\substack{\chi \shortmod q\\ \chi^3=\chi_0}} \Big|
\sumstar\limits_{\substack{M<m\le 2M}} a_m \chi(m)\Big|^2
\ll
 \Delta(Q,M) \sumstar\limits_{\substack{M<m\le 2M\\ (m,3)=1}}\left| a_m \right|^2,
\end{equation}
where the star at the sum over $m$ indicates that it is taken over squarefree integers and
\begin{equation} \label{Deltabound}
\Delta(Q,M) = (QM)^{\varepsilon} \min\{Q^{5/3}+M,Q^{4/3}+Q^{1/2}M,Q^{11/9}+Q^{2/3}M, Q+Q^{1/3}M^{5/3}+M^{12/5}\}.
\end{equation}
\end{mytheo}

For comparison, note that the ordinary large sieve gives the following weaker results
\begin{gather*}
\sum_{q \leq Q} \; \sumstar_{\substack{\chi \shortmod{q} \\ \chi^3 = \chi_0}} |L(\tfrac12, \chi)|^8 \ll Q^{2 + \varepsilon},
\qquad
\sum_{q \leq Q} \; \sumstar_{\substack{\chi \shortmod{q} \\ \chi^3 = \chi_0}} |L(\tfrac12, \chi)|^2 \ll Q^{\frac54 + \varepsilon},
\\
\sum_{q \leq Q} \; \sumstar_{\substack{\chi \shortmod{q} \\ \chi^3 = \chi_0}} |L(\tfrac12, \chi)| \ll Q^{\frac98 + \varepsilon}.
\end{gather*}
The $8$th moment is deduced simply by embedding the family of cubic characters into the family of all Dirichlet characters of conductor $\leq Q$, and using the known bound \eqref{eightmomentbound}. The estimates for the first and second moments follow by Cauchy's inequality.

 We point to \cite{Elliott}, Section 7, for some early large sieve-type results on general $r$-th order characters.
 Related results to Theorem \ref{cubiclargesieve} are Heath-Brown's quadratic and cubic large sieves \cite{Hea0} \cite{Hea}. The quadratic large sieve states
\begin{equation}
\label{realfinal}
 \sumflat_{|d| \leq Q} \big| \sumstar_{m \leq M} a_m \chi_d(m) \big|^2 \ll (Q + M)(QM)^{\varepsilon} \sumstar_{m \leq M}|a_m|^2,
\end{equation}
where the sum over $d$ runs over fundamental discriminants, and $\chi_d$ is the associated primitive quadratic character. 
The cubic large sieve states
\begin{equation}
\label{eq:HBcubic}
 \sumstar_{\substack{n \in \mz[\omega] \\N(n) \leq N}} \big| \sumstar_{\substack{m \in \mz[\omega] \\N(m) \leq M}} a_m \leg{m}{n}_3 \big|^2 \ll (M + N + (MN)^{2/3})(MN)^{\varepsilon} \sumstar_{m \leq M} |a_m|^2,
\end{equation}
where the stars indicate that $m,n$ run over squarefree elements of $\mz[\omega]$ that are congruent to $1 \pmod 3$.
Both of these results are proved with a recursive use of Poisson summation.
Our method of proof of Theorem \ref{cubiclargesieve} uses \eqref{eq:HBcubic} (after some transformations), and avoids 
recursion.  One of the new difficulties with treating cubic Dirichlet characters is the asymmetry between $\chi$ and $m$.  
It takes some calculation to see that a direct application of \eqref{eq:HBcubic}, choosing $a_m$ to have support on rational integers, 
is not better than \eqref{Deltabound}.  Note that $m \leq M$ means $N(m) \leq M^2$ so that \eqref{eq:HBcubic} implies $\Delta(Q,M) \ll (QM)^{\varepsilon}(Q + M^2 + (QM^2)^{2/3})$.  For example, $M = \sqrt{Q}$ gives here $\Delta(Q,\sqrt{Q}) \ll Q^{4/3 + \varepsilon}$ while the fourth bound of \eqref{Deltabound} gives $\Delta(Q, \sqrt{Q}) \ll Q^{6/5 + \varepsilon}$, which in fact is the key to proving \eqref{secondm}.

It is of great interest to extend these results to higher-order characters and to different number fields.  In general we wish to understand these families of twists in the Katz-Sarnak sense \cite{KSM} \cite{KS}.  One obvious analytic issue is the degree of the field extension $\mq(e^{2\pi i/l})/\mq$, as discussed following equation \eqref{eq:2} above.  It is plausible that our methods could generalize to $l=4, 6$ since for these cases this degree is also $2$.  However, for the application of the central values of $L$-functions, we also require estimates for the sum of $l$-th order Gauss sums, which become somewhat worse as $l$ increases (e.g. see Proposition 1 of \cite{P2}).

We can generalize some of our results to sextic characters.
\begin{mytheo} \label{sexticlargesieve}
Theorems \ref{secondmoment} and \ref{cubiclargesieve} remain valid when the condition $\chi^3=\chi_0$ is replaced by the weaker condition $\chi^6=\chi_0$.
\end{mytheo}
\noindent The proof is nearly identical to those of Theorems \ref{secondmoment} and \ref{cubiclargesieve} so we omit the details.

A motivating example for this generalization to sextic twists is to understand the behavior of the family of elliptic curves
\begin{equation*}
 y^2 = x^3 + b,
\end{equation*}
where $b \in \mz$.  These curves have complex multiplication by $\mq(\omega)$, and have $L$-functions that can be expressed using the sextic residue symbol $(4b/n)_6$, where $n$ runs over elements of $\mz[\omega]$.  The study of this family of $L$-functions clearly leads to double sums of the form addressed in Theorem \ref{sexticlargesieve}.\\


{\bf Acknowledgements.} The second-named author would like to thank B. Brubaker and S.J. Patterson for useful comments. Parts of this work were done when the first-named author visited Texas A\&M University in November 2007 and September 2009. He wishes to thank this institution for the invitation, its warm hospitality during his pleasant stays, and for financial support. Both authors would like to thank A. Diaconu for helping us to understand his work \cite{DiTi} with Y. Tian, and to an anonymous referee for pointing out that \cite{DiTi} may be relevant to our investigations.

\section{Preliminaries}
In this section we provide various tools used throughout the paper.
\subsection{Properties of the cubic characters}
\label{section:cubicclassification} 
The cubic characters are related to the arithmetic of the quadratic number field $\mq(\omega)$, $\omega = e^{2 \pi i/3}$, with ring of integers $\mz[\omega]$ and discriminant $-3$.  This field has class number one and has six units, $\pm \{1, \omega, \omega^2\}$, and one ramified prime $1-\omega$ dividing $3$.  Each principal ideal $0 \neq (n) \subset \mz[\omega]$ with $(n,3) = 1$ has a unique generator $n \equiv 1 \pmod{3}$; sometimes we may implicitly choose such a generator.

\begin{mylemma}
\label{lemma:cubicclass}
 The primitive cubic Dirichlet characters of conductor $q$ coprime to $3$ are of the form $\chi_n:m \rightarrow (\frac{m}{n})_3$ for some $n \in \mz[\omega]$, $n \equiv 1 \pmod{3}$, $n$ squarefree and not divisible by any rational primes, with norm $N(n) = q$.
\end{mylemma}
\noindent This analysis can also be found in \cite{David} e.g. but we shall present this here for completeness.   We refer to Chapter 9 of \cite{IR} for basic properties of cubic residues.

\begin{proof}
To classify the characters we use an approach similar to that of \cite{Davenport}, Chapter 5. By multiplicativity, it suffices to consider the case that $q=p^a$ is a prime power.  It is not hard to show that there is a primitive character of conductor $p$ if and only if $p \equiv 1 \pmod{3}$, in which case there are exactly two such characters $\pmod{p}$, each being the square of the other. If $a \geq 2$ and $p \neq 3$ then there is no primitive character of order $3$, as any character of order $3$ must be induced from one $\mymod{p}$.

These cubic characters are intimately connected with the cubic residue symbol in the ring $\mz[\omega]$. 
If $p \equiv 1 \pmod{3}$ then $p = \pi \overline{\pi}$ with $N(\pi) = p$, and then there is an associated cubic character $l \rightarrow (\frac{l}{\pi})_3$.  This cubic character is defined by the conditions $\leg{l}{\pi}_3 \equiv l^{\frac{N(\pi)-1}{3}} \pmod{\pi}$, with $\leg{l}{\pi}_3 \in \{ 1, \omega, \omega^2 \}$.  It follows directly from the definition that for $l \in \mz$, $(\frac{l}{\overline{\pi}})_3 = \overline{(\frac{l}{\pi})_3}$.  Thus for prime conductor $p$ there is a one-to-one correspondence between primitive Dirichlet characters of order $3$ and conductor $p$, and cubic residue symbols $\chi_{\pi}$ with $N(\pi) = p$.  By multiplicativity we extend this one-to-one correspondence to squarefree $q$ and elements $n$ of $\mz[\omega]$ such that $N(n) = q$.  It is easy to see that $N(n)$ is squarefree (as an element of $\mz$) if and only if $n$ is squarefree (as an element of $\mz[\omega]$) and $n$ has no rational prime divisor.
\end{proof}

Recall that the cubic reciprocity law states that for $m, n \in \mz[\omega]$, $m, n \equiv \pm 1 \pmod{3}$,
\begin{equation*}
 \leg{m}{n}_3 = \leg{n}{m}_3.
\end{equation*}
Some sources state this for $m \equiv n \equiv -1 \pmod{3}$, but the fact that $(-1)^3 = -1$ easily allows for this slight generalization.
The supplement states that if $\pi = 1 + 3a + 3b \omega$, where $a,b \in \mz$, then
\begin{equation*}
 \leg{1-\omega}{\pi}_3 = \omega^{a}.
\end{equation*}
Note $3 = -\omega^2(1-\omega)^2$, and that $N(\pi) = (1+3a)^2 - (1+3a)3b + 9b^2 \equiv 1+3a +3b \pmod{9}$.  Therefore, a simple calculation shows
\begin{equation*}
 \leg{\omega}{\pi}_3 = \omega^{2a + 2b},
\end{equation*}
and hence
\begin{equation*}
 \leg{3}{\pi}_3 = \omega^{b}.
\end{equation*}
It follows easily that for any $n \equiv 1 \pmod{3}, n \in \mz[\omega]$, written in the form $n =  1 + 3c + 3d \omega$, then $(3/n)_3 = \omega^{d}$.  In particular, for $n \equiv 1 \pmod{3}$, we have $(3/n)_3 = 1$ if and only if $n \equiv 1, 4, 7 \pmod{9}$, and in general $(3/n)_3$ only depends on $n \pmod{9}$.  The functions $n \rightarrow ((1-\omega)/n)_3$ and $n \rightarrow (\omega/n)_3$ are ray class characters $\mymod{9}$.

Suppose that $m \in \mz[\omega]$, $m$ not a cube nor a unit.  Then the function
$\psi_m:(n) \rightarrow \leg{m}{n}_3$
defined on ideals $(n) \subset \mz[\omega]$ coprime to $3$, where $n \equiv 1 \pmod{3}$, gives a class group character of modulus $9m$.  
Hence (see Theorem 12.5 of \cite{I}) the Hecke $L$-function
\begin{equation*}
 L(s,\psi_m) = \sum_{(n)} \psi_m((n)) N(n)^{-s} = \sum_{\substack{n \in \mz[\omega] \\ n \equiv 1 \shortmod{3}}} \leg{m}{n}_3 N(n)^{-s}
\end{equation*}
is associated to a weight one cusp form of level $3 N(9m)$ and nebentypus character $\chi(a) = (-3/a) \psi_m((a))$, where $(-3/a)$ is the Kronecker symbol.  

\subsection{On the Gauss sums}
\label{section:Gauss}
It turns out that the Gauss sum associated to the Dirichlet character $\chi_n$ (on $\mz$) defined in Lemma \ref{lemma:cubicclass} is the same one as the corresponding Hecke character (on $\mz[\omega]$).  We now prove this important fact.  Recall the definition of the standard Gauss sum for $n \in \mz[\omega]$, $n \equiv 1 \pmod{3}$ (as in \cite{H-BP} for instance),
\begin{equation*}
 g(n) = \sum_{d \shortmod{n}} \leg{d}{n}_3 \check{e}(d/n), \quad \text{where} \quad \check{e}(z) = \exp(2 \pi i(z + \overline{z})).
\end{equation*}
Our notation differs from \cite{H-BP} as we reserve $e(z)$ for the more standard $\exp(2 \pi i z)$.
Recall that 
$n$ has no rational prime divisor, so $(n, \bar{n}) = 1$. By definition,
\begin{equation*}
 \tau(\chi_n) =  \sum_{1 \leq x \leq N(n)} \left(\frac{x}{n}\right)_3 e^{\frac{2 \pi i x}{N(n)}}.
\end{equation*}
Now write $x \equiv y \bar{n} + \bar{y} n \pmod{n \overline{n}}$, where $y$ varies over a set of representatives in $\mz[\omega] \pmod{n}$, and here $\bar{n}$ is the complex conjugate of $n$.  It is easy to see that as $y$ varies $\pmod{n}$, $x$ varies $\pmod{N(n)}$, using that $\bar{x} = x$ and the Chinese Remainder Theorem.  We then see
\begin{equation*}
 \tau(\chi_n) =  \sum_{y \shortmod{n}} \left(\frac{y \bar{n}}{n}\right)_3 e^{2 \pi i (\frac{y}{n} + \overline{\frac{y}{n}})}.
\end{equation*}
It is a consequence of cubic reciprocity that $(\frac{\bar{n}}{n})_3 = 1$ for $n \equiv 1 \pmod{3}$, so we have
\begin{equation}
\label{eq:taug}
 \tau(\chi_n) = g(n).
\end{equation}

Now we collect some basic facts on the Gauss sums.
It is well-known that
\begin{equation}
\label{eq:gcubed}
 g(n)^3 = \mu(n) N(n) n,
\end{equation}
whence one derives the pleasant fact that $g(n)$ vanishes unless $n$ is squarefree.  

Generalize the definition of the Gauss sum by setting
\begin{equation*}
 g(r,n) = \sum_{x \shortmod{n}} \leg{x}{n}_3 \check{e}\leg{rx}{n}.
\end{equation*}
See \cite{H-BP}, pp.123-124 for the following formulas.  First,
\begin{equation}
\label{eq:gmult}
 g(rs,n) = \overline{\leg{s}{n}_3} g(r,n), \qquad \text{if } (s,n) = 1.
\end{equation}
Furthermore, if $(n_1, n_2) = 1$ then
\begin{equation}
\label{eq:gtwisted}
 g(r,n_1 n_2) = \overline{\leg{n_1}{n_2}_3} g(r,n_1) g(r,n_2)
 = g(n_2 r, n_1) g(r, n_2).
\end{equation}
We also compute, for $\pi$ prime in $\mz[\omega]$, $k \geq 1$,
\begin{equation}
\label{eq:gramanujan}
 g(\pi^2, \pi^k) =
\begin{cases}
	-N(\pi^2), \qquad & k=3, \\
	0, \qquad & \text{otherwise}.
\end{cases}
\end{equation}
In addition, we shall require the fact that
\begin{equation}
\label{eq:gissquarefree}
 g(r,n) = 0, \qquad \text{if } \pi^2 | n, \pi \nmid r,
\end{equation}
which follows immediately from \eqref{eq:gcubed} and \eqref{eq:gmult}.

\begin{mylemma}
\label{lemma:gausssumcalculation}
Suppose $n_1, n_2, \delta \in \mz[\omega]$ are squarefree, $\equiv 1 \pmod{3}$, with norms that are pairwise relatively prime.  
Then
\begin{equation}
\label{eq:gausssumcalculation}
\tau(\chi_{n_1} \overline{\chi_{n_2 \delta}}) = \left(\frac{n_2 \delta}{\overline{n_1}}\right)_3 \left(\frac{\delta}{n_2}\right)_3
\tau(\chi_{n_1}) \tau(\overline{\chi_{n_2}})\tau(\overline{\chi_{\delta}}).
\end{equation}
\end{mylemma}
\begin{proof}
The conditions ensure that $\chi_{n_1} \overline{\chi_{n_2 \delta}}$ is a primitive character.
It follows from the definition of the cubic residue symbol that $\overline{\leg{m}{n}}_3 = \leg{\overline{m}}{\overline{n}}_3$, so for $m \in \mz$, $\overline{\chi_n}(m) = \chi_{\overline{n}}(m)$.  Thus $\tau(\chi_{n_1} \overline{\chi_{n_2 \delta}}) = \tau(\chi_{n_1 \overline{n_2 \delta}})$.  Repeatedly using \eqref{eq:taug}, \eqref{eq:gtwisted}, and cubic reciprocity, we get \eqref{eq:gausssumcalculation}.
\end{proof}

\subsection{The approximate functional equation}
Using an approximate functional equation (see Theorem 5.3 of \cite{IK}), we have
\begin{myprop} 
\label{prop:AFE}
Let $\chi$ be an odd primitive Dirichlet character $\chi$ of conductor $q$, and make the following definitions: Let
\begin{equation*}
 V_{\alpha}(x) = \frac{1}{2\pi i} \int_{(1)} \frac{G(s)}{s} g_{\alpha}(s) x^{-s} ds, \quad \text{where} \quad g_{\alpha}(s) = \pi^{-s/2} \frac{\Gamma\left(\tfrac{\frac{3}{2} + \alpha+ s}{2}\right)}{\Gamma\left(\tfrac{\frac{3}{2} + \alpha}{2}\right)}.
\end{equation*}
Furthermore, let $\epsilon(\chi) = i^{-1} q^{-1/2} \tau(\chi)$ be essentially the (normalized) Gauss sum 
and set
\begin{equation*}
 X_{\alpha} = \left(\frac{q}{\pi}\right)^{-\alpha} \frac{\Gamma\left(\tfrac{\frac{3}{2} - \alpha}{2}\right)}{\Gamma\left(\tfrac{\frac{3}{2} + \alpha}{2}\right)}.
\end{equation*}
Finally let $A$ and $B$ be positive real numbers such that $AB = q$.  Then for any $|\text{Re}(\alpha)| < \half$ we have
\begin{equation} \label{approxfunc}
L(\tfrac12 + \alpha, \chi) = \sum_{m=1}^{\infty} \frac{\chi(m)}{m^{\frac12 + \alpha}} V_{\alpha}\left(\frac{m}{A}\right) + \epsilon(\chi) X_{\alpha} \sum_{m=1}^{\infty} \frac{\overline{\chi}(m)}{m^{\frac12 - \alpha}} V_{-\alpha}\left(\frac{m}{B}\right).
\end{equation}
\end{myprop}
\noindent For $\alpha = 0$ we set $V_0 = V$.

\subsection{Poisson summation}
We shall require two versions of the Poisson summation formula.  Suppose that $w$ is a smooth, compactly-supported function on the positive reals.  

Let $\chi$ be a primitive Dirichlet character of conductor $q$.  Then
\begin{equation}
\label{eq:Poisson1dim}
 \sum_{m \in \mz} w\leg{m}{M} \chi(m) = \frac{M}{q} \tau(\chi) \sum_{h \in \mz} \overline{\chi}(h) \widehat{w}\leg{hM}{q}.
\end{equation}
This is well-known.  For the latter version, we directly quote Lemma 10 of \cite{Hea}.  Let $\chi(m) = \leg{m}{n_1}_3 \overline{\leg{m}{n_2}_3}$ where $n_1$ and $n_2$ are elements of $\mz[\omega]$ that are coprime to each other, and to $3$, and are squarefree.  Then $\chi$ is a primitive character on $\mz[\omega]$ of modulus $n_1 n_2$.
\begin{mylemma} \label{2dpoisson}
 We have
\begin{equation*}
 \sum_{m \in \mz[\omega]} w\leg{N(m)}{M} \chi(m) = \frac{\chi(\sqrt{-3}) g(n_1) \overline{g(n_2)} M}{N(n_1 n_2)} \sum_{k \in \mz[\omega]} \overline{\chi}(k) \check{w}\left(\sqrt{\frac{N(k)}{N(n_1 n_2)} M} \right),
\end{equation*}
where
\begin{equation*}
 \check{w}(t) = \intR \intR w(N(x+y \omega)) \check{e}(t(x+y\omega)/\sqrt{-3}) dx dy.
\end{equation*}
\end{mylemma}

\section{Proof of Theorem \ref{thm:mainLvalueresult}}
Here we give the proof of Theorem \ref{thm:mainLvalueresult} though a variant on the bound \eqref{eq:secondmomentHecke} used in a critical way is proved in Section \ref{section:secondmomentproof}.  
We first note that by Lemma \ref{lemma:cubicclass},
\begin{equation*} \label{splitting}
\mathcal{M} := \sum_{(q,3)=1} \ \sumstar_{\substack{\chi \shortmod{q} \\ \chi^3 = \chi_0}} L(\tfrac12, \chi) w\leg{q}{Q}
=
\sumprime_{n \equiv 1 \shortmod{3}} L(\tfrac12, \chi_{n}) w\left(\frac{N(n)}{Q}\right)
\end{equation*}
where the prime indicates the sum runs over squarefree elements $n$ of $\mathbb{Z}[\omega]$ that have no rational prime divisor.  
Applying the approximate functional equation, Proposition \ref{prop:AFE}, with $A_n B = N(n)$ gives $\mathcal{M} = \mathcal{M}_1 + \mathcal{M}_2$, where
\begin{equation*}
 \mathcal{M}_1 = \sumprime_{n \equiv 1 \shortmod{3}}\ \sum_{m=1}^{\infty} \frac{\chi_n(m)}{\sqrt{m}} V\leg{m}{A_n} w\left(\frac{N(n)}{Q}\right)
\end{equation*}
and
\begin{equation*}
 \mathcal{M}_2 = \sumprime_{n \equiv 1 \shortmod{3}} \epsilon(\chi_n) \sum_{m=1}^{\infty} \frac{\overline{\chi}_n(m)}{\sqrt{m}} V\leg{m}{B} w\left(\frac{N(n)}{Q}\right).
\end{equation*}
Define $A$ by $AB=Q$ so that $A_n = A \frac{N(n)}{Q} \asymp A$ for all $n$ under consideration, in view of the support of $w$.

We shall treat $\mathcal{M}_1$ and $\mathcal{M}_2$ with different methods.  The results are summarized with
\begin{mylemma}
 We have
\begin{equation}
\label{eq:M1estimate}
 \mathcal{M}_1 = c Q \widetilde{w}(1) + O(Q^{1/2 + \varepsilon} A^{3/4} + Q A^{-1/6 + \varepsilon}),
\end{equation}
and
\begin{equation}
\label{eq:M2bound}
 \mathcal{M}_2 \ll Q^{5/6} B^{1/6}  + Q^{2/3} B^{5/6}.
\end{equation}
\end{mylemma}
\noindent Choosing $B = Q^{7/19}$, whence $A = Q^{12/19}$  gives Theorem \ref{thm:mainLvalueresult}.  The constant $c$ is given more explicitly in Section \ref{section:mainterm} below.

Our approach for $\mathcal{M}_1$ employs the summation over $n$ to transform the expression into one involving Hecke $L$-functions.  Then we bound this new expression with \eqref{eq:Heckevariant} which is a consequence of Theorem \ref{cubiclargesieve}.  In a previous version of this paper (available on the arxiv), we set up a complicated recursive technique that was later used (with other ingredients) by the second author in a simpler setting \cite{Y2}.
The cancellation in $\mathcal{M}_2$ comes from the sum of cubic Gauss sums, which follows from the work of Patterson showing that these cubic Gauss sums appear as Fourier coefficients of metaplectic Eisenstein series \cite{P}.  See Lemma \ref{lemma:sumofGauss} below for the estimate on the sum of cubic Gauss sums.

\subsection{Evaluating $\mathcal{M}_1$}
\label{section:M1}
First we work on $\mathcal{M}_1$.  We shall detect the condition that $n\equiv 1 \pmod{3}$ has no rational prime divisor using the formula
\begin{equation}
\label{eq:ratmob}
 \sum_{\substack{d |n, d \in \mz \\ d \equiv 1 \shortmod{3}}} \mu_{\mz}(d) =
\begin{cases}
 1, \quad \text{$n$ has no rational prime divisor}, \\
 0, \quad \text{otherwise}.
\end{cases}
\end{equation}
Here we define $\mu_{\mz}(d) = \mu(|d|)$, the usual M\"{o}bius function.  The choice of $d$ up to unit, namely $d \equiv 1 \pmod{3}$ is natural for the arithmetic of the ring $\mz[\omega]$.  We apply this formula and change variables $n \rightarrow dn$ to the sum over $n$.  Since $d$ is squarefree as an element of $\mz[\omega]$, the condition that $dn$ is squarefree then simply means that $n$ is squarefree and $(d,n) = 1$.  Thus
\begin{equation*}
 \mathcal{M}_1 = \sum_{\substack{d \in \mz \\ d \equiv 1 \shortmod{3} }} \mu_{\mz}(d) \sum_{m=1}^{\infty} \frac{\leg{m}{d}_3}{\sqrt{m}} \sumstar_{\substack{n \equiv 1 \shortmod{3} \\ (n,d) = 1}} \leg{m}{n}_3 V\left(\frac{m}{A} \frac{Q}{N(nd)} \right) w\left(\frac{N(nd)}{Q}\right).
\end{equation*}
Now we use M\"{o}bius inversion again (writing $\mu_{\omega}(l)$ for the M\"{o}bius function on $\mz[\omega]$) to detect the condition that $n$ is squarefree, getting
\begin{equation*}
 \mathcal{M}_1 = \sum_{\substack{d \in \mz \\ d \equiv 1 \shortmod{3} }} \mu_{\mz}(d) \sum_{l \equiv 1 \shortmod{3}} \mu_{\omega}(l) \sum_{m=1}^{\infty} \frac{\leg{m}{d l^2}_3}{\sqrt{m}} \mathcal{M}_1(d,l,m),
\end{equation*}
where
\begin{equation*}
 \mathcal{M}_1(d,l,m) = \sum_{\substack{n \equiv 1 \shortmod{3} \\ (n,d) = 1}} \leg{m}{n}_3 V\left(\frac{m}{A} \frac{Q}{N(ndl^2)} \right) w\left(\frac{N(ndl^2)}{Q}\right).
\end{equation*}
Next we use the Mellin transform of the weight function to express the sum over $n$ as a contour integral involving the Hecke $L$-function.  By Mellin inversion,
\begin{equation*}
 V\left(\frac{m}{A} \frac{Q}{N(ndl^2)} \right) w\left(\frac{N(ndl^2)}{Q}\right) = \frac{1}{2 \pi i} \int_{(2)} \leg{Q}{N(ndl^2)}^s \widetilde{f}(s) ds,
\end{equation*}
where
\begin{equation*}
\widetilde{f}(s) = \int_0^{\infty} V\left(\frac{m}{A} x\right) w(x) x^{s-1} dx.
\end{equation*}
Integration by parts shows $\widetilde{f}(s)$ is a function satisfying the bound for all $\text{Re}(s) \geq \frac14$
\begin{equation*}
 \widetilde{f}(s) \ll (1 + |s|)^{-100} (1 + m/A)^{-100}.
\end{equation*}
With this notation, and with the definition \eqref{eq:HeckeL}, then
\begin{equation*}
 \mathcal{M}_1(d,l,m) = \frac{1}{2 \pi i} \int_{(2)} \leg{Q}{N(dl^2)}^s L(s, \psi_m) \widetilde{f}(s) ds.
\end{equation*}
We estimate $\mathcal{M}_1$ by moving the contour to the half line.  When $m$ is a cube the Hecke $L$-function has a pole at $s=1$.  We set $\mathcal{M}_0$ to be the contribution to $\mathcal{M}_1$ of these residues, and $\mathcal{M}_1'$ to be the remainder.  We shall defer the analysis of $\mathcal{M}_0$ to Section 
\ref{section:mainterm}.

By bounding everything with absolute values, we see that
\begin{equation*}
|\mathcal{M}_1'| \ll  \sum_{d \ll \sqrt{Q}} \sum_{N(l) \ll \sqrt{Q}} \frac{1}{\sqrt{N(dl^2)}} \sum_{m} \frac{\sqrt{Q}}{\sqrt{m}} (1+m/A)^{-100} \intR |L(\tfrac12 + it, \psi_m)| (1+|t|)^{-100} dt.
\end{equation*}
Since $d \in \mz$, then $N(d) = d^2$ so that the sums over $d$ and $l$ are easily computed.
Finally we use the estimate \eqref{eq:Heckevariant}, which is a close relative to \eqref{eq:secondmomentHecke}, to bound the sum over $m$.
Putting everything together, we obtain
\begin{equation}
\label{eq:M1'}
|\mathcal{M}_1'| \ll Q^{1/2 + \varepsilon} A^{3/4},
\end{equation}
In Section \ref{section:mainterm} we show $\mathcal{M}_0 = cQ \widetilde{w}(1) + O(QA^{-1/6 + \varepsilon})$ which combined with \eqref{eq:M1'} gives \eqref{eq:M1estimate}.

\subsection{Computing $\mathcal{M}_0$}
\label{section:mainterm}
Recall 
that
\begin{equation*}
 \mathcal{M}_0 = \sum_{\substack{d \in \mz \\ d \equiv 1 \shortmod{3} }} \mu_{\mz}(d) \sum_{l \equiv 1 \shortmod{3}} \mu_{\omega}(l) \sum_{m=1}^{\infty} \frac{\leg{m}{d l^2}_3}{\sqrt{m}} \frac{Q}{N(dl^2)} \widetilde{f}(1) \text{Res}_{s=1} L(s, \psi_m),
\end{equation*}
where using the Mellin convolution formula shows
\begin{equation*}
 \widetilde{f}(1) = \int_0^{\infty} V\left(\frac{m}{A} x\right) w(x) dx = \frac{1}{2 \pi i} \int_{(1)} \leg{A}{m}^s \widetilde{w}(1-s) \frac{G(s)}{s} g(s) ds.
\end{equation*}	
From the discussion in Section \ref{section:cubicclassification}, it is not difficult to see that $\psi_m$ is the principal character only if $m$ is a cube, in which case
\begin{equation*}
 L(s, \psi_m) = \zeta_{\mq(\omega)}(s) \prod_{\pi | 3m} (1 - N(\pi)^{-s}),
\end{equation*}
and $\zeta_K(s)$ is the Dedekind zeta function for the field $K$.  Let $c_\omega = \frac{2 \pi}{6 \sqrt{3}}$ be the residue of $\zeta_{\mq(\omega)}(s)$ at $s=1$, evaluated using the Kronecker limit formula.  Then
\begin{equation*}
 \mathcal{M}_0 = c_\omega Q \sum_{m=1}^{\infty} \frac{\widetilde{f}(1)}{m^{3/2}} \prod_{\pi | 3m} (1- N(\pi)^{-1})  \sum_{\substack{d \in \mz, (d,m) =1\\ d \equiv 1 \shortmod{3} }} \frac{\mu_{\mz}(d)}{d^2} \sum_{\substack{(l,m)=1 \\ l \equiv 1 \shortmod{3}}} \frac{\mu_{\omega}(l)}{N(l^2)}.
\end{equation*}
Computing the sums over $d$ and $l$ explicitly, we obtain
\begin{equation*}
 \mathcal{M}_0 = c_\omega Q \sum_{m=1}^{\infty} \frac{\widetilde{f}(1)}{m^{3/2}} \prod_{\pi | 3m} (1- N(\pi)^{-1})  \prod_{p \nmid 3m} (1-p^{-2}) \prod_{\pi \nmid 3m} (1- N(\pi)^{-2}).
\end{equation*}
The two products over $\pi$ combine rather nicely to give
\begin{equation*}
 \mathcal{M}_0 = c_\omega  \zeta^{-1}_{\mq(\omega)}(2) \zeta^{-1}(2) Q \sum_{m=1}^{\infty} \frac{\widetilde{f}(1)}{m^{3/2}}   \prod_{\pi | 3m} (1+ N(\pi)^{-1})^{-1} \prod_{p | 3m} (1-p^{-2})^{-1}.
\end{equation*}
Let
\begin{equation*}
 Z(u) = \sum_{m=1}^{\infty} m^{-u} \prod_{\pi | 3m} (1+ N(\pi)^{-1})^{-1} \prod_{p | 3m} (1-p^{-2})^{-1},
\end{equation*}
which is holomorphic and bounded for $\text{Re}(u) \geq 1 + \delta > 1$.  Then
\begin{equation*}
 \mathcal{M}_0 = c_\omega  \zeta^{-1}_{\mq(\omega)}(2) \zeta^{-1}(2) Q \frac{1}{2 \pi i} \int_{(1)} A^s Z(\tfrac32 + 3s) \widetilde{w}(1-s) \frac{G(s)}{s} g(s) ds.
\end{equation*}
We move the contour of integration to $-1/6 + \varepsilon$, crossing a pole at $s=0$ only.  The new contour contributes $O(A^{-1/6 + \varepsilon} Q)$, while the pole at $s=0$ gives
\begin{equation}
\label{eq:c}
c Q \widetilde{w}(1), \quad \text{where} \quad c= c_\omega  \zeta^{-1}_{\mq(\omega)}(2) \zeta^{-1}(2) Z(3/2).
\end{equation}
Note that $Z(u)$ converges absolutely at $u=3/2$ so it is easy to express $Z(3/2)$ explicitly as an Euler product, if desired.

\subsection{Estimating $\mathcal{M}_2$}
\label{section:M2}
Using the calculation $\epsilon(\chi_n) = i^{-1} g(n) N(n)^{-\half}$, we have
\begin{equation*}
\label{eq:M2recall}
 \mathcal{M}_2 = i^{-1} \sum_{m=1}^{\infty} \frac{1}{m^{\tfrac12}} V\left(\frac{m}{B}\right) \sumprime_{\substack{ n\equiv 1 \shortmod{3}}} \frac{\overline{\chi}_{n}(m) g(n)}{\sqrt{N(n)}} w\left(\frac{N(n)}{Q}\right) .
\end{equation*}
In this section we show
\begin{mylemma}
\label{lemma:sumofGauss}
 For any $m \in \mz[\omega]$, write $m = m_0 m_1$ where $m_0$ is a unit times a power of $1-\omega$ and $m_1 \equiv 1 \pmod{3}$.  Then we have
\begin{equation*}
\label{eq:sumofGausssums}
H'(n,Q):=\sumprime_{\substack{n \in \mz[\omega] \\ n\equiv 1 \shortmod{3}}} \frac{\overline{\chi}_n(m) g(n) }{\sqrt{N(n)}} w\left(\frac{N(n)}{Q}\right) \ll  Q^{2/3 + \varepsilon} N(m)^{1/6} + Q^{5/6} N(m_1)^{-1/6 + \varepsilon}.
\end{equation*}
\end{mylemma}
By summing trivially over $m$ one easily deduces \eqref{eq:M2bound}.
Recall that the prime on the sum over $n$ indicates that the sum is restricted to squarefree numbers having no rational prime divisor.  This feature causes some difficulties.

Our first move in the proof of Lemma \ref{lemma:sumofGauss} is to use
M\"{o}bius inversion, i.e., \eqref{eq:ratmob}, to remove the condition that $n$ has no rational prime divisor.  We simplify the resulting expression using the the identity $g(dn) = g(d) g(n) \overline{\chi}_n(d)$ following from \eqref{eq:gtwisted}, the fact that $g(n) = 0$ unless $n$ is squarefree, and using the notation $\widetilde{g}(c) = g(c) N(c)^{-\half}$.
This gives
\begin{equation*}
\label{eq:M2sieved}
 H'(n,Q) = \sum_{\substack{d \in \mz \\ d \equiv 1 \shortmod{3}}} \mu_{\mz}(d) \tilde{g}(d) H(dm, Q/d^2),
\end{equation*}
where
\begin{equation*}
H(dm, X) = \sum_{\substack{n \in \mz[\omega] \\ n \equiv 1 \shortmod{3}}}  \frac{\overline{\chi}_{n}(dm) g(n)}{N(n)^{\frac12}} w\left(\frac{N(n)}{X}\right).
\end{equation*}
We estimate $H$ with the following
\begin{mylemma}
\label{lemma:Hbound}
 For any $l \in \mz[\omega]$, write $l = l_0 l_1$ where $l_0$ is a unit times a power of $1-w$, and $l_1 \equiv 1 \pmod{3}$.  Then we have
\begin{equation*}
\label{eq:Hbound}
 H(l, X) \ll X^{1/2 + \varepsilon} N(l_1)^{1/4} + X^{5/6} N(l_1)^{-1/6 + \varepsilon}.
\end{equation*}
\end{mylemma}
Before proving Lemma \ref{lemma:Hbound}, we show how Lemma \ref{lemma:sumofGauss} follows from it.  We treat $|d| \leq Y$ and $|d| > Y$ separately, where $Y$ is a parameter to be chosen.  For $|d| \leq Y$ we use Lemma \ref{lemma:Hbound}, while for $|d| > Y$ we use the trivial bound $H(l,X) \ll X$.  Thus, writing $m = m_0 m_1$ where $m_0$ is a power of $3$ and $m_1$ is coprime to $3$, we have
\begin{equation*}
 H'(n,Q) \ll \sum_{|d| \leq Y} \leg{Q}{d^2}^{1/2+\varepsilon} N(dm)^{1/4} + \sum_{|d| \leq Y} \leg{Q}{d^2}^{5/6} N(m_1)^{-1/6+\varepsilon} + \sum_{|d| > Y} \frac{Q}{d^2},
\end{equation*}
which simplifies as
\begin{equation*}
 H'(n,Q) \ll Q^{1/2 + \varepsilon} \sqrt{Y} N(m)^{1/4} + Q Y^{-1} + Q^{5/6} N(m_1)^{-1/6 + \varepsilon}.
\end{equation*}
Optimally choosing $Y = Q^{1/3} N(m)^{-1/6}$ gives Lemma \ref{lemma:sumofGauss}.

\begin{proof}[Proof of Lemma \ref{lemma:Hbound}]
The difficulty in estimating $H(l,X)$ is apparently a technicality: the sum is not quite in a form that allows us to quote from the literature, in particular Section 4 of \cite{H-BP}.  Our goal is to manipulate $H(l,X)$ until it meets these conditions.  Before elaborating on this discussion we first do some minor simplifications that ease the comparison to the literature.

In this section we use the convention that all sums over elements of $\mz[\omega]$ are restricted to elements $\equiv 1 \pmod{3}$.  Writing $l = l_0 l_1$ as above, and
using cubic reciprocity, we see that
\begin{equation*}
\overline{\chi}_n(l) := \overline{\leg{l}{n}}_3 = \overline{\leg{n}{l_1}}_3 \overline{\leg{l_0}{n}}_3. 
\end{equation*}
From the discussion in Section \ref{section:cubicclassification}, the function $\lambda(n) = \overline{\leg{l_0}{n}}_3$ is a ray class character $\pmod{9}$.  
Thus
\begin{equation*}
H(l,X) = \sum_{\substack{ n \in \mz[\omega]}}  \frac{\lambda(n) \overline{\leg{n}{l_1}}_3 g(n)}{N(n)^{\frac12}} w\left(\frac{N(n)}{X}\right).
\end{equation*}
Note that the identity \eqref{eq:gmult} implies $\overline{\leg{n}{l_1}} g(n) = g(l_1, n)$ for $(n, l_1) = 1$.  Letting
\begin{equation*}
 h(r,s) = \sum_{(n,r) = 1} \frac{\lambda(n) g(r,n)}{N(n)^s},
\end{equation*}
and introducing the Mellin transform of $w$, we get
\begin{equation}
\label{eq:HlX}
 H(l,X) =   
\frac{1}{2 \pi i} \int_{(2)} \widetilde{w}(s) X^s h(\tfrac12 + s, l_1) ds.
\end{equation}
Clearly the series defining $h(r,s)$ converges absolutely and uniformly on any region $\text{Re}(s) \geq \frac32 + \delta > \frac32$.  We need to know the analytic behavior of $h(r,s)$, i.e., meromorphic continuation, location of poles, and order of growth.  In Section 4 of \cite{H-BP} these properties are explicitly given but for slightly different functions, such as
\begin{equation*}
 \psi(r,s) = \sum_n \frac{g(r,n)}{N(n)^{s}}, \quad \psi_{\alpha}(r,s) = \sum_{n \equiv 0 \shortmod{\alpha}} \frac{g(r,n)}{N(n)^{s}}, \quad \widetilde{\psi}_{\alpha}(r,s) = \sum_{(n,\alpha) = 1} \frac{g(r,n)}{N(n)^s}.
\end{equation*}
Precisely, with the following we summarize results from Lemma 4 of \cite{H-BP}, and Theorems 9.1 and 8.1 of \cite{P2}.
\begin{mylemma}
\label{lemma:psi}
 The function $\psi(r,s)$ has meromorphic continutation to the complex plane.  It is holomorphic in the region $\text{Re}(s) > 1$ except possibly for a pole at $s=4/3$.  Furthermore, letting $\sigma_1 = \frac32 + \varepsilon$, and $\sigma_1 \geq \sigma \geq \sigma_1 - \frac12$, $|s-\frac43| > \frac{1}{12}$, we have
\begin{equation*}
 \psi(r,s) \ll N(r)^{\frac12 (\sigma_1 - \sigma)} (1 + t^2)^{\sigma_1 - \sigma}.
\end{equation*}
If $r= r_1 r_2^2$ is cubefree, then the residue satisfies
\begin{equation*}
\text{res}_{s=4/3} \psi(r,s) \ll N(r_1)^{-1/6+\varepsilon}.
\end{equation*}
\end{mylemma}
Note $h(r,s)$ differs from $\psi(r, s)$ only in the additional presence of the ray class character $\lambda$, and the coprimality condition $(n,r) = 1$.  The presence of $\lambda$ is unimportant, but removing the condition $(n,r) = 1$ unfortunately seems to require some elaborate gyrations.

\begin{mylemma}
\label{lemma:h}
Lemma \ref{lemma:psi} holds with $\psi(r,s)$ replaced by $h(r,s)$.
\end{mylemma}
Before proving Lemma \ref{lemma:h} we show how it implies Lemma \ref{lemma:Hbound}.
  We move the line of integration in \eqref{eq:HlX} to $\text{Re}(s) = \frac12 + \varepsilon$, crossing a pole at $s =5/6$, which contributes
\begin{equation*}
\ll X^{5/6} N(l_1)^{-1/6 + \varepsilon}.
\end{equation*}
The main contribution comes from the new line of integration, which gives
\begin{equation*}
\ll  X^{1/2+\varepsilon} N(l_1)^{1/4}.
\end{equation*}
This completes the proof of Lemma \ref{lemma:Hbound}.
\end{proof}

\begin{proof}[Proof of Lemma \ref{lemma:h}]
We shall use inclusion-exclusion type arguments to reduce the estimation of $h(r,s)$ to that of $\psi(r,s)$.  To this end, we collect some of these results with the following
\begin{mylemma}
\label{lemma:laundrylist}
 Suppose $f$, $\alpha$ are squarefree and $(r,f) = 1$, and set
\begin{equation*}
 h(r,f,s) = \sum_{(n,rf) = 1} \frac{\lambda(n) g(r,n)}{N(n)^s}, \quad h_{\alpha}(r,s) = \sum_{(n,\alpha) =1} \frac{\lambda(n) g(r,n)}{N(n)^s}.
\end{equation*}
Furthermore suppose $r = r_1 r_2^2 r_3^3$ where $r_1 r_2$ is squarefree, and let $r_3^*$ be the product of primes dividing $r_3$.
Then
\begin{align}
\label{eq:3.6.1}
 &h(r,f,s) = \sum_{a | f} \frac{\mu_{\omega}(a) \lambda(a) g(r,a)}{N(a)^{s}} h(ar,s), \\
\label{eq:3.6.2}
&h(r_1 r_2^2 r_3^3, s) = h(r_1r_2^2, r_3^*,s), \\
\label{eq:3.6.3}
&h(r_1 r_2^2, s) = \prod_{\pi | r_2} (1 - \lambda(\pi)^3 N(\pi)^{2-3s})^{-1} 
h_{r_1}(r_1 r_2^2, s),
\end{align}
\begin{equation}
\label{eq:3.6.4}
h_{r_1}(r_1 r_2^2, s) = \prod_{\pi | r_1}(1 -\lambda(\pi)^3 N(\pi)^{2-3s})^{-1} \sum_{a | r_1} \mu_{\omega}(a) N(a)^{1-2s} \lambda(a)^2 \overline{g(r_1 r_2^2/a, a)} h_1(r_1r_2^2/a, s).
\end{equation}
\end{mylemma}
Before embarking on the technical details of this proof, we show how it proves Lemma \ref{lemma:h}.  The function $h_1(r, s)$ is identical to $\psi(r,s)$ except it is twisted by $\lambda(n)$, the ray class character of modulus $9$.  Then $h_1$ satisfies the properties of Lemma \ref{lemma:psi}, the necessary generalizations having been carried out in \cite{P2} for example.  By working backwards and using (\ref{eq:3.6.1}-\ref{eq:3.6.4}), we see that $h(r,s)$ has meromorphic continuation and potential pole at $s=4/3$ only, and
\begin{equation*}
 h(r,f,s) \ll N(f)^{\varepsilon} N(r)^{\half(\sigma_1 - \sigma)} (1+t^2)^{\sigma_1 - \sigma}.
\end{equation*}
The analogous bound on $h(r,s)$ follows.  The estimate on the residue follows by a similar method.
\end{proof}

\begin{proof}[Proof of Lemma \ref{lemma:laundrylist}] 
Using M\"{o}bius to remove the condition $(n,f) = 1$ gives
\begin{equation*}
 h(r,f,s) = \sum_{a | f} \frac{\muK(a) \lambda(a)}{N(a)^s} \sum_{(n,r) = 1} \frac{\lambda(n) g(r,an)}{N(n)^s}.
\end{equation*}
Notice that if $\pi |a$ then $\pi \nmid r$ so if in addition $\pi | n$ then by \eqref{eq:gissquarefree}, $g(r,an) = 0$.  Thus we may assume $(n,a) = 1$, in which case $g(r,an) = g(ar,n) g(r,a)$ by \eqref{eq:gtwisted}, and hence \eqref{eq:3.6.1} holds.

From \eqref{eq:gmult} it follows that $g(r_1 r_2^2 r_3^3,n) = g(r_1 r_2^2, n)$ provided $(n, r_3) = 1$, whence \eqref{eq:3.6.2} holds.

Now we prove \eqref{eq:3.6.3}.  For this we introduce some new notation as follows.  Let $ab^2 \in \mz[\omega]$ and let $\pi$ be prime such that $(ab,\pi) = 1$.  Then
\begin{equation*}
 h_2(a\pi^2, b^2,s) := \sum_{(n,a\pi) = 1} \frac{\lambda(n) g(a b^2 \pi^2,n)}{N(n)^s} = \sum_{(n,a) = 1} \frac{\lambda(n) g(a b^2\pi^2,n)}{N(n)^s} - \sum_{(n,a) = 1, \pi | n} \frac{\lambda(n)g(a b^2\pi^2,n)}{N(n)^s}.
\end{equation*}
Writing in the latter sum $n = \pi^j n'$, where $(n', \pi) = 1$, then we have $g(a b^2 \pi^2, \pi^{j} n') = g(\pi^{j+2}a b^2, n') g(ab^2 \pi^2, \pi^{j})$ by \eqref{eq:gtwisted}.  Using \eqref{eq:gmult} we get $g(a b^2\pi^2, \pi^{j}) = \overline{(ab^2/\pi^{j})_3} g(\pi^2, \pi^{j})$, which is nonzero if and only if $j=3$, from \eqref{eq:gramanujan}.  Thus we get $g(a b^2\pi^2, \pi^3 n') = -N(\pi^2) g(a b^2 \pi^2, n')$.  In summary, we have shown
\begin{equation*}
 h_2(s,a\pi^2, b^2) = \sum_{(n,a) = 1} \frac{\lambda(n)g(a b^2 \pi^2,n)}{N(n)^s} + \lambda(\pi)^3 N(\pi)^{2-3s} h_2(s,a\pi^2, b^2),
\end{equation*}
which when rearranged states
\begin{equation*}
 \sum_{(n, a\pi) =1} \frac{\lambda(n) g(a b^2 \pi^2, n)}{N(n)^s} = (1- \lambda(\pi^3) N(\pi)^{2-3s})^{-1} \sum_{(n,a) = 1} \frac{\lambda(n) g(a b^2 \pi^2, n)}{N(n)^s}.
\end{equation*}
An induction argument on the number of prime divisors of $b$ gives
\begin{equation*}
 h(r_1 r_2^2,s) = \prod_{\pi | r_2} (1 - \lambda(\pi)^3 N(\pi)^{2 - 3s})^{-1} \sum_{(n,r_1)=1} \frac{\lambda(n) g(r_1 r_2^2, n)}{N(n)^s},
\end{equation*}
which is the same as \eqref{eq:3.6.3}.
%
Finally, the relation \eqref{eq:3.6.4} 
is a slight generalization of Lemma 3(i) of \cite{H-BP}, the only difference being that the sums in \eqref{eq:3.6.4} are twisted by $\lambda(n)$.  Since $\lambda(n)$ is completely multiplicative, an inspection of the argument of \cite{H-BP} easily shows that the proof generalizes to give \eqref{eq:3.6.4}.
\end{proof}

\section{The cubic large sieve}
In this section we establish our cubic large sieve, Theorem \ref{cubiclargesieve}.
It is easy to reduce the expression in question, namely the left hand side of \eqref{final}, to a sum of similar expressions with the additional summation conditions $(q,3)=1$ and $(m,3)=1$ included. Thus it suffices to estimate
\begin{equation*} \label{trans}
\sum\limits_{\substack{Q<q\le 2Q\\ (q,3)=1}}\ \sumstar\limits_{\substack{\chi \shortmod q\\ \chi^3=\chi_0}} \Big|
\sumstar\limits_{\substack{M<m\le 2M\\ (m,3)=1}} a_m \chi(m)\Big|^2
= 
\sumprime\limits_{\substack{n\in \mz[\omega]\\ Q<N(n)\le 2Q\\ n\equiv 1 \shortmod 3}} \Big|
\sumstar\limits_{\substack{M<m\le 2M\\ (m,3)=1}} a_m\chi_n(m) \Big|^2 =: 
T(Q,M),
\end{equation*}
where the prime indicates that $n$ is squarefree and has no rational prime divisor. 

Throughout this section, we follow the conventions that $n$ denotes an element of $\mz[\omega]$, that $m$ is a rational integer and that the coefficients $a_m$ are supported at integers $m$ coprime to $3$ in the interval $(M,2M]$.
The reader should recall that 
$\chi_n(m)=\left(\frac{m}{n}\right)_3$, defined for any $m,n \in \mz[\omega]$ with $n \equiv  1 \pmod{3}$.  Note that $\chi_n(m) = \chi_m(n)$ for all $m$ and $n$ appearing in the definition of $T(Q,M)$.

The primary goal of this section is to estimate the expression $T(Q,M)$. To this end, we will frequently make use of ideas and results in \cite{Hea0} and \cite{Hea}, where \eqref{realfinal} and \eqref{eq:HBcubic} were established, respectively (in particular, we shall use \eqref{eq:HBcubic} itself). However, here we have to manage the additional difficulty lying in the asymmetry of the sums over $m$ and $n$. This will require some new ideas. In particular, we shall use H\"older's inequality to enlarge the sum over $m$ and {\it two} versions of the Poisson summation formula: the one-dimensional version for the sum over $m\in \mz$ and the two-dimensional version for the sum over $n\in \mz[\omega]$.

\subsection{Definition of certain norms}
\label{section:normdefs}
In the following, we define several norms which we later compare and estimate.
We begin by defining a norm corresponding to $T(Q,M)$ by
\begin{equation} \label{B1def}
B_1(Q,M):=\sup\limits_{(a_m)} ||a_m||^{-2} \sumprime\limits_{\substack{Q<N(n)\le 2Q\\ n\equiv 1 \shortmod 3}} \Big|
\sumstar\limits_{m} a_m \chi_m(n)\Big|^2,
\quad \text{where} \quad
||a_m||^2 = \sum_{m} |a_m|^2,
\end{equation}
and where by convention we suppose that $(a_m)$ is not identically zero.  Note that we used cubic reciprocity for this formulation.  We recall that the prime at the outer sum indicates that $n$ is squarefree and has no rational prime divisor.

We further define a norm $B_2(Q,M)$ in the same way as $B_1(Q,M)$ except removing the condition that $n$ has a rational prime divisor.  Similarly, we define a norm $B_3(Q,M)$ by further removing the condition that $n$ is squarefree.

Let $W:\mr \rightarrow \mr$ be a fixed smooth, nonnegative, compactly-supported function such that $W(x) \geq 1$ for $1 \leq x \leq 2$.
It follows that $B_3(Q,N)$ is bounded by
\begin{equation*}
B_3(Q,N)\le \sup\limits_{(a_m)} ||a_m||^{-2} \sum\limits_{n}
W\left(\frac{N(n)}{Q}\right) \Big|
\sumstar\limits_{m} a_m \chi_m(n)\Big|^2.
\end{equation*}
Expanding the square and rearranging the summation, the right-hand side takes the form  
\begin{equation*} \label{db3bound}
\sup\limits_{(a_m)} ||a_m||^{-2}\sumstar\limits_{m_1, m_2} a_{m_1}\overline{a_{m_2}}
\sum\limits_{n}
W\left(\frac{N(n)}{Q}\right) \chi_{m_1}(n)\overline{\chi_{m_2}}(n).
\end{equation*}
As in \cite{Hea}, it will turn out that we may restrict attention to the case in which $m_1$ and $m_2$ are coprime. We define another norm $B_4$ corresponding to the above sum with the restriction $(m_1,m_2)=1$ included by
\begin{equation*}
B_4(Q,M):=\sup\limits_{(a_m)} ||a_m||^{-2} \sumstar\limits_{(m_1,m_2)=1} a_{m_1}\overline{a_{m_2}}
\sum\limits_{n}  W\left(\frac{N(n)}{Q}\right) \chi_{m_1}(n)\overline{\chi_{m_2}}(n).
\end{equation*}

We further define a norm $C_1(M,Q)$ dual to $B_1(Q,M)$ by
\begin{equation*}
C_1(M,Q):=\sup\limits_{(b_{n})} ||b_n||^{-2} \sumstar\limits_{\substack{M<m\le 2M\\ (m,3)=1}} \Big| \sumprime\limits_{n} b_{n} \chi_n(m)\Big|^2,
\end{equation*}
where here as in the sequel, we assume that the coefficients $b_n$ are supported at elements $n$ of $\mz[\omega]$ with $Q<N(n)\le 2Q$ and $n\equiv  1 \pmod 3$.
We further recall that the star at the outer sum indicates that $m$ is squarefree. By the duality principle, $C_1(M,Q)=B_1(Q,M)$. 

Finally, we define a norm $C_2(M,Q)$ by extending the summation over $m$ in the definition of $C_1(M,Q)$ to all integers $m$ with $M<m\le 2M$.

\subsection{Proof of Theorem \ref{cubiclargesieve}}
\label{section:comparisonC}
We begin by collecting various properties that the norms satisfy.  

\begin{mylemma} \label{normlemma} Let $Q,M\ge 1$ and $C$ be a sufficiently large positive constant. Then we have the following inequalities:
\begin{equation} \label{C2e1}
C_2(M,Q) \ll (QM)^{\varepsilon}\left(M + Q^{5/3}\right);
\end{equation}
\begin{equation} \label{C22}
C_2(M,Q)\ll M^{\varepsilon}Q^{1-1/v}\sum\limits_{j=0}^{v-1}C_2(2^jM^v,Q)^{1/v}, \quad \text{ for each fixed positive integer $v$};
\end{equation}
\begin{equation} \label{B11}
B_1(Q_1,M)\ll B_1(Q_2,M), \quad \text{ if $Q_1,M\ge 1$ and $Q_2\ge CQ_1\log (2Q_1M)$};
\end{equation}
\begin{equation} \label{B21}
B_2(Q,M) \ll (\log{2Q})^3 Q^{1/2}X^{-1/2} B_1(XQ^{\varepsilon},M), \quad \text{for some $X$ with $1\le X\le Q$};
\end{equation}
\begin{equation} \label{B32}
B_3(Q,M) \ll (\log 2Q)^3 Q^{1/2}X^{-1/2} B_2(XQ^{\varepsilon},M), \quad \text{for some $X$ with $1\le X\le Q$};
\end{equation}
\begin{equation} \label{B34}
B_3(Q,M)\ll M^{\varepsilon} B_4\left(\frac{Q}{\Delta_1},\frac{M}{\Delta_2}\right), \quad 
\text{for some $\Delta_1,\Delta_2\in \mathbb{N}$ with $\Delta_2^2\ge \Delta_1$};
\end{equation}
\begin{multline} \label{B43}
B_4(Q,M)\ll Q+QM^{\varepsilon-2} \max\left\{B_3(K,M) : K\le M^4Q^{-1}\right\}
\\
+ Q^{-1}M^{6+\varepsilon} \sum\limits_{K> M^4/Q} K^{-2-\varepsilon} B_3(K,M),
\end{multline}
where the sum over $K$ in \eqref{B43} runs over powers of $2$.
\end{mylemma}

We postpone the proof of Lemma \ref{normlemma} to the following sections and now deduce Theorem \ref{cubiclargesieve}. Recall that we need to prove that
\begin{equation} \label{whatweneed}
B_1(Q,M)\ll (QM)^{\varepsilon}\min \{Q^{5/3}+M,Q^{4/3}+Q^{1/2}M,Q^{11/9}+Q^{2/3}M, Q+Q^{1/3}M^{5/3}+M^{12/5}\}.
\end{equation}
The first estimate in the minimum follows from \eqref{C2e1} and the trivial bound $B_1(Q,M)=C_1(M,Q)\le C_2(M,Q)$. The second and the third estimate are obtained by combining \eqref{C2e1} and \eqref{C22}, with $v=2,3$, and then using $B_1(Q,M)\le C_2(M,Q)$. 

All that remains is to show the last inequality in \eqref{whatweneed}.  This bound is most relevant for the second moment of cubic Dirichlet $L$-functions, i.e.,  \eqref{secondm}.  For this, we use the relations between the various norms.  Specifically,
 we shall start with the already-established bound
\begin{equation} \label{C2v3}
B_1(Q,M)\ll (QM)^{\varepsilon}\left(Q^{11/9}+Q^{2/3}M\right)
\end{equation}
(third term in the minimum in \eqref{whatweneed}) as an initial estimate, deduce bounds for $B_3$ and $B_4$ from it and then work backwards, obtaining new bounds for $B_3$ and finally $B_1$.
In details, we begin by combining \eqref{B21} with \eqref{C2v3} to get 
\begin{equation*}
 B_2(Q,M) \ll (QM)^{\varepsilon} Q^{1/2} X^{-1/2} (X^{11/9} + X^{2/3} M).
\end{equation*}
 The worst case is $X=Q$ which shows $B_2(Q,M)$ also satisfies \eqref{C2v3}.  Repeating the argument, we have
\begin{equation*}
B_3(Q,M)\ll (QM)^{\varepsilon}(Q^{11/9}+Q^{2/3}M).
\end{equation*}
Combining this with \eqref{B43}, we obtain
\begin{eqnarray*}
B_4(Q,M)&\ll& Q+(QM)^{\varepsilon}QM^{-2} \max\left\{K^{11/9}+K^{2/3}M \ :\ K\le M^4Q^{-1}\right\}\\ & & + (QM)^{\varepsilon}M^{6}Q^{-1} \sum\limits_{K\ge M^4/Q} K^{-2-\varepsilon}(K^{11/9}+K^{2/3}M)\nonumber\\
&\ll& Q+(QM)^{\varepsilon}(Q^{-2/9}M^{26/9}+Q^{1/3}M^{5/3}).\nonumber
\end{eqnarray*}
From this and \eqref{B34}, we deduce that
\begin{equation*}
B_3(Q,M)\ll \frac{Q}{\Delta_1}+(QM)^{\varepsilon}\left(
\left(\frac{Q}{\Delta_1}\right)^{-2/9}  \left(\frac{M}{\Delta_2}\right)^{26/9}+
\left(\frac{Q}{\Delta_1}\right)^{1/3}  \left(\frac{M}{\Delta_2}\right)^{5/3}\right)
\end{equation*}
for some positive integers $\Delta_1$, $\Delta_2$ with $\Delta_2^2\ge \Delta_1$. 
The worst case is $\Delta_2 = \Delta_1 = 1$.
Using this together with the trivial bound $B_1(Q,M)\le B_3(Q,M)$ gives
\begin{equation*} \label{tem1}
B_1(Q,M)\ll Q+(QM)^{\varepsilon}\left(Q^{-2/9}M^{26/9}+Q^{1/3}M^{5/3}\right).
\end{equation*}
This bound can in general be improved by taking $Q$ larger, so we use the increasing property \eqref{B11}
to replace $Q$ by 
$Q^{1+\varepsilon}+M^{11/5}
$
which gives the desired bound
\begin{equation*} \label{Re2}
B_1(Q,M)\ll (QM)^{\varepsilon} \left(Q+Q^{1/3}M^{5/3}+M^{12/5}\right).
\end{equation*}
This completes the proof of Theorem \ref{cubiclargesieve}. $\Box$ \\

We remark that a further cycle in the above process does not lead to an improvement of our result.

For convenience, we enclose a table displaying the estimates for $B_1(Q,M)$ that we get for various ranges. This table should be read as follows. If the fractions $\alpha$ and $\beta$ are the $(n-1)$-th and $n$-th entries, respectively, in the first row, and the term $T$ is the $n$-th entry in the second row, then the estimate $B_1(Q,M)\ll (QM)^{\varepsilon}T$ holds in the range $M^{\alpha}<Q\le M^{\beta}$. 

\begin{center}
\begin{tabular}{|r||c|c|c|c|c|c|c|c|c|} \hline
  \rule[-.1cm]{0cm}{.8cm} Range 
  &\makebox[1cm][c]{$\frac{3}{5}$}&\makebox[1cm][c]{$\frac{6}{7}$}&\makebox[1cm][c]{$\frac{6}{5}$}
  &\makebox[1cm][c]{$\frac{3}{2}$} &\makebox[1cm][c]{$\frac{9}{5}$}&\makebox[1cm][c]{$\frac{108}{55}$}
  &\makebox[1cm][c]{$\frac{11}{5}$}&\makebox[1cm][c]{$\frac{5}{2}$}&\makebox[1cm][c]{$\infty$}\\ \hline
  \rule[-.1cm]{0cm}{.8cm} Bound  & $M$ & $Q^{5/3}$ & $Q^{1/2}M$ & $Q^{4/3}$ & $Q^{2/3}M$ & $Q^{11/9}$ & $M^{12/5}$ & $Q^{1/3} M^{5/3}$ & $Q$  \\\hline
 \end{tabular}
\end{center}

\subsection{Proof of Theorem \ref{secondmoment}}
\label{section:secondmomentproof}
Before we turn to the proof of Lemma \ref{normlemma}, we establish Theorem \ref{secondmoment} which is in fact an easy consequence of Theorem \ref{cubiclargesieve}. 
Since all steps are standard, we will only sketch the arguments.

We first establish \eqref{secondm}.  
Using \eqref{approxfunc}, the approximate functional equation, with $A=B=\sqrt{q}$ and $\alpha=it$, and Cauchy's inequality, we estimate the second moment in question by
\begin{equation*}
\sum\limits_{\substack{q\le Q}}  \; \sumstar\limits_{\substack{\chi \shortmod q\\ \chi^3=\chi_0}} \left| L(1/2+it,\chi) \right|^2 \leq 2 \sum\limits_{\substack{q\le Q}} \; \sumstar\limits_{\substack{\chi \shortmod q\\ \chi^3=\chi_0}} \left| \sum_{m=1}^{\infty} \frac{\chi(m)}{m^{\tfrac12 + it}} V_{it}\left(\frac{m}{\sqrt{q}}\right) 
\right|^2.
\end{equation*}
The analytic conductor of $L(1/2 + it, \chi)$ is $\asymp q (1 + |t|)$ so that by Proposition 5.4 of \cite{IK}, $V_{it}(x) \ll_R (1+ x(1+|t|)^{-1/2})^{-R}$ for any $R > 0$.  Thus we may truncate $m$ so that $m \leq M:=(Q(1+|t|))^{1/2 + \varepsilon}$ with a negligibly small error.

Then we break the summations over $q$ and $m$ into dyadic intervals and remove the weight $V_{it}$ using the Mellin transform. We further write $m=d^2n$, where $n$ is squarefree, and use the Cauchy-Schwarz inequality again. Eventually, we arrive at sums of the form
\begin{equation*} \label{end}
\sum\limits_{d\le \sqrt{2M}} \frac{1}{d} \sum\limits_{\substack{Q<q\le 2 Q}}\ \sumstar\limits_{\substack{\chi \shortmod q\\ \chi^3=\chi_0}} \left| \sumstar_{M/d^2<m\le 2M/d^2} \frac{\chi(m)}{m^{\tfrac12 + it}} \right|^2
\end{equation*}
which we then estimate by using Theorem \ref{cubiclargesieve}.  More precisely, we use
\eqref{Deltabound} with the last term, $Q+Q^{1/3}M^{5/3}+M^{12/5}$, in the minimum.  Plugging this bound in and summing trivially over $d$ gives \eqref{secondm}.

Next we establish \eqref{eq:HeckeL}.  For $m$ squarefree, the character $\psi_m(n) = \leg{m}{n}_3$ defined on $n \equiv 1 \pmod{3}$ is primitive with conductor $\mathfrak{f}$ satisfying $m/(3,m) | \mathfrak{f}$, $\mathfrak{f} | 9m$.  Thus the Hecke $L$-function $L(s, \psi_m)$, viewed as a degree $2$ $L$-function over $\mq$, has conductor $\ll N(m) (1+ t^2) = m^2(1+t^2)$.  A variant on the above argument reduces the problem of estimating \eqref{eq:secondmomentHecke} to bounding
\begin{equation*}
 \sumstar_{m \leq M} \Big|\sumstar_{N(n) \ll Q} \frac{\chi_n(m)}{N(n)^{\half + it}} \Big|^2,
\end{equation*}
where $Q \ll  (M(1+|t|))^{1+\varepsilon}$.  The bound $C_1(M,Q) \ll (QM)^{\varepsilon}(Q^{4/3} + Q^{1/2} M)$ from Theorem \ref{cubiclargesieve} then gives the desired estimate.  In the course of the proof of Theorem \ref{thm:mainLvalueresult} we actually require the following variant
\begin{equation}
\label{eq:Heckevariant}
 \sum_{m \leq M} \frac{1}{\sqrt{m}} |L(1/2 + it, \psi_m)| \ll M^{3/4 + \varepsilon} (1+|t|)^{2/3 + \varepsilon}.
\end{equation}
To prove this version, we factor $m$ as $m_1 m_2^2 m_3^3$ where $(m_1, m_2) = 1$.  Then $\psi_m$ equals $\psi_{m_1} \overline{\psi_{m_2}}$ times a principal character.  For each fixed $m_2$, we then generalize \eqref{eq:HeckeL} to give
\begin{equation*}
 \sumstar_{m_1 \leq M_1, (m_1, m_2)=1} |L(1/2 + it, \psi_{m_1} \overline{\psi_{m_2}}|^2 \ll M_1^{3/2+\varepsilon} m_2^{4/3+\varepsilon} (1+|t|)^{2/3+\varepsilon}.
\end{equation*}
With this bound and a use of Cauchy's inequality, it is easy to sum over $m_2$ trivially, giving \eqref{eq:Heckevariant}.

\subsection{Proof of Lemma \ref{normlemma}, estimate \eqref{C2e1}}
\label{section:C2}
The key point in establishing our cubic large sieve is the estimation of the norm $C_2(M,Q)$, which we do in this subsection. We point out that the ordinary large sieve inequality gives only the weaker bound $C_2(M,Q) \ll M + Q^{2}$.

Recall that $C_2(M,Q)$ is the norm associated to the sum 
\begin{equation*}
\label{eq:C2def}
S(M,Q):=\sum\limits_{M<m\le 2M} \Big| \sumprime\limits_{n}
 b_{n} \chi_n(m)\Big|^2,
\end{equation*}
where the prime indicates that $n$ is squarefree and has no rational prime divisor.
The sum $S(M,Q)$ is obviously bounded by
\begin{equation*}
\label{eq:C2defW}
S(M,Q)\leq S_W(M,Q):=\sum\limits_{m\in \mz} W\left(\frac{m}{M}\right) \Big| \sumprime\limits_{n} b_{n} \chi_n(m)\Big|^2 ,
\end{equation*}
where the weight function $W$ is defined as in Section \ref{section:normdefs} 
Expanding the square and rearranging the summation, we get
\begin{equation*}
S_W(M,Q)=\sumprime \limits_{n_1,n_2}
b_{n_1}\overline{b_{n_2}}
\sum\limits_{m\in \mz}
W\left(\frac{m}{M}\right) \chi_{n_1}\overline{\chi_{n_2}}(m).
\end{equation*}

Now the idea is to use the Poisson summation formula to transform the inner sum over $m\in \mathbb{Z}$. This will eventually lead us to an expression that can be bounded by directly using Heath-Brown's cubic large sieve inequality \eqref{eq:HBcubic}. However, before applying Poisson summation, it will be convenient to reduce our characters $\chi_{n_1}\overline{\chi_{n_2}}(m)$ to primitive characters. To achieve this, we need to extract the greatest common divisor $\Delta$ of $n_1$ and $n_2$ as well as the greatest common divisor $\delta$ of $n_1$ and $\overline{n_2}$. Extracting $\Delta$, we get
\begin{equation*}
S_W(M,Q)=
\sumprime_{\substack{\Delta,n_1,n_2 \\ (n_1, n_2) = 1}}
b_{n_1 \Delta}\overline{b_{n_2 \Delta}}
\sum\limits_{(m, N(\Delta)) = 1}
W\left(\frac{m}{M}\right) \chi_{n_1}\overline{\chi_{n_2}}(m).
\end{equation*}
Next, we extract the greatest common divisor $\delta$ of
$n_1$ and $\overline{n_2}$, changing variables via $n_1 \rightarrow \delta n_1$, $n_2 \rightarrow \overline{\delta} n_2$.  The coprimality conditions become $(n_1, \overline{n_2}) = 1$ and $(\delta n_1, \overline{\delta} n_2) = 1$.  From these conditions, combined with the facts that $n_1 \delta$ and $n_2 \overline{\delta}$ are squarefree and have no rational prime divisor, we see that $(N(n_1), N(n_2 \delta)) = 1$.  Thus
\begin{equation*}
\label{eq:prePoisson}
S_W(M,Q)=
\sumprime_{\substack{\Delta,\delta,n_1,n_2 
\\ (N(n_1), N(n_2 \delta)) = 1 }}
b_{n_1 \Delta \delta}\overline{b_{n_2 \Delta \overline{\delta}}}
\sum\limits_{(m, N(\Delta)) = 1}
W\left(\frac{m}{M}\right) \chi_{n_1}\overline{\chi_{n_2 \delta}}(m),
\end{equation*}
where we use that $\chi_{\delta}\overline{\chi_{\overline{\delta}}}=\chi_{\delta}^2=
\overline{\chi_{\delta}}$.  
We still need to remove the coprimality condition in the sum over $m$ before we can apply Poisson summation. Doing this by using the M\"obius function, we get
\begin{equation}
\label{eq:prePoisson2}
S_W(M,Q)=
\sumprime_{\substack{\Delta,\delta,n_1,n_2 \\ (N(n_1), N(n_2 \delta) = 1 }}
b_{n_1 \Delta \delta}\overline{b_{n_2 \Delta \overline{\delta}}}
\sum_{l | N(\Delta)} \mu(l) \chi_{n_1}\overline{\chi_{n_2 \delta}}(l)  \sum\limits_{m}
W\left(\frac{m}{M/l}\right) \chi_{n_1}\overline{\chi_{n_2 \delta}}(m).
\end{equation}

Now the characters $\chi_{n_1}\overline{\chi_{n_2 \delta}}(m)$ in the above expression are primitive, and we have a smooth sum over $m$. Applying the Poisson summation formula in the form given in \eqref{eq:Poisson1dim}, 
we have
\begin{equation}\label{afterpoisson}
\sum\limits_{m \in \mz}
W\left(\frac{m}{M/l}\right) \chi_{n_1}\overline{\chi_{n_2 \delta}}(m) =
\frac{M \tau(\chi_{n_1} \overline{\chi_{n_2 \delta}})}{l N(n_1 n_2 \delta)}  \sum\limits_{h \in \mz} \overline{\chi_{n_1}}\chi_{n_2\delta}(h) 
\widehat{W}\left(\frac{hM}{lN(n_1 n_2 \delta)}\right).
\end{equation}
When $h=0$, then the summand above is zero unless $n_1 = n_2 = \delta =1$. Hence, the contribution of $h=0$ to $S_W(M,Q)$, say $S_0(M,Q)$ satisfies
\begin{equation*} \label{h0cont}
S_0(M,Q) \ll M^{1+\varepsilon} \sumprime_{\Delta}
|b_{\Delta}|^2 \ll M^{1+\varepsilon}||b||^2.
\end{equation*}

Let $S'_{W}(M,Q)$ be the contribution to $S_W(M,Q)$ from $h \neq 0$.  We analyze $S'$ now, where we need to show
\begin{equation}
\label{eq:SW'bound}
S'_{W}(M,Q) \ll Q^{5/3} (QM)^{\varepsilon} \sum_n |b_n|^2.
\end{equation}
Our strategy is to apply Heath-Brown's cubic large sieve estimate \eqref{eq:HBcubic}.  We need the following easy consequence of \eqref{eq:HBcubic}: Suppose that $d_n, d_n'$ are arbitrary complex numbers supported on squarefree $n \in \mz[\omega]$, $n \equiv 1 \pmod{3}$, with $N(n) \leq X$.  Then by Cauchy's inequality,
\begin{equation}
\label{eq:HBcubic2}
|\sum_{m, n} d_m d_n' \leg{m}{n}_3| \leq (\sum_m |d_m|^2)^{1/2} (\sum_{m} | \sum_n d_n' \leg{n}{m}_3|^2)^{1/2} \ll X^{2/3 + \varepsilon} ||d_m|| \cdot ||d_n'||.
\end{equation}

\begin{proof}[Proof of \eqref{eq:SW'bound}]
First observe that we
may freely truncate the sum over $h$ for
\begin{equation*}
|h| \le \frac{Q^2 l}{N(\delta)N(\Delta)^2M} (QM)^{\varepsilon}=:H,
\end{equation*}
since $\widehat{W}$ has rapid decay.  Precisely, if we let $S_W'(M,Q) = S_W''(M,Q)+E$ where $S_W''(M,Q)$ is the contribution to $S_W'(M,Q)$ from $0<|h|\leq H$, then $E \ll (MQ)^{-100} ||b||^2$.
Further, we note that the relevant range for $n_1,n_2$ is $N(n_1),N(n_2) \asymp Q/N(\Delta\delta)$ since the coefficients $b_n$ are supported at $n\in \mz[\omega]$ with $N(n) \asymp Q$.  
Combining \eqref{eq:prePoisson2}, \eqref{afterpoisson}, using Lemma \ref{lemma:gausssumcalculation}, and changing variables $n_1 \rightarrow \overline{n_1}$, we arrive at the following bound
\begin{equation} \label{hnot0}
S_W''(M,Q) \ll 
M \sumprime_{\Delta} \sumprime_{\delta} \frac{1}{N(\delta)^{1/2}} \sum_{l | \Delta} \frac{1}{l} \sum_{0<|h|\le H} \left| U(\Delta,\delta,l,h) \right|,
\end{equation}
where 
\begin{equation*}
U(\Delta,\delta,l,h):=\sumprime\limits_{\substack{N(n_1),N(n_2) \asymp Q/N(\Delta\delta)\\ (N(n_1), N(n_2)) = 1 }} \widehat{W}\left(\frac{hM}{lN(n_1 n_2 \delta)}\right) c_{\Delta,\delta,l,h}(n_1) c_{\Delta,\delta,l,h}'(n_2) \left(\frac{n_1}{n_2}\right)_3,
\end{equation*}
and the coefficients $c,c'$ satisfy the bounds
\begin{equation*}
c_{\Delta,\delta,l,h}(n)\ll \left(\frac{N(\delta\Delta)}{Q}\right)^{1/2} |b_{\overline{n}\Delta\delta}|, \quad c_{\Delta,\delta,l,h}'(n) \ll  \left(\frac{N(\delta\Delta)}{Q}\right)^{1/2} |b_{n\Delta\overline{\delta}}|.
\end{equation*}

Now we are almost ready to use Heath-Brown's cubic large sieve inequality in the form \eqref{eq:HBcubic2} to bound the sum $U(\Delta,\delta,l,h)$.  The only obstacle is that the variables $n_1$ and $n_2$ are not separated due to the coprimality condition $(N(n_1),N(n_2))=1$ and the weight function $\widehat W$.  This is only a technical obstacle since one can use M\"{o}bius inversion to remove the coprimality condition, and the Mellin inversion formula to remove the weight function, both at essentially no cost.  Hence
\begin{equation*}
U(\Delta,\delta,l,h) \ll (QM)^{\varepsilon}\left(\frac{N(\delta \Delta)}{Q}\right)^{1/3}
\sumprime_{n} |b_{n}|^2.
\end{equation*}
Inserting this into \eqref{hnot0} and summing trivially over all the other variables gives \eqref{eq:SW'bound}.
\end{proof}

\subsection{Proof of Lemma \ref{normlemma}, estimate \eqref{C22}}
To prove the self-referential estimate \eqref{C22} for $C_2(M,Q)$, we introduce a dual norm
\begin{equation} \label{C2dual}
C_2'(Q,M):=\sup\limits_{(a_m)} ||a_m||^{-2} \sumprime\limits_{\substack{Q<N(n)\le 2Q\\ n\equiv  1 \shortmod 3}} \Big|
\sum\limits_{m} a_m \chi_n(m) \Big|^2.
\end{equation}
By the duality principle, we have $C_2'(Q,M)=C_2(M,Q)$.
Assume $(a_{m})$ is a sequence such that the supremum in \eqref{C2dual} is attained. 
Then, by H\"older's inequality and multiplicativity of the residue symbol, we get
\begin{equation*} \label{hol}
C_2'(Q,M) \ll  ||a_m||^{-2}  Q^{1-1/v}
\Big(\sumprime\limits_{\substack{Q<N(n)\le 2Q\\ n \equiv \pm 1 \shortmod 3}} \Big|
\sum\limits_{\substack{M^v<m\le (2M)^v}} c_m \chi_n(m) \Big|^{2}
\Big)^{1/v},
\end{equation*}
where
\begin{equation*}
c_m=\sum\limits_{m_1\cdots m_v=m} a_{m_1}\cdots a_{m_v}.
\end{equation*}
By splitting the sum over $m$ into dyadic segments, we have
\begin{equation} \label{normbound}
C_2'(Q,M) \ll Q^{1-1/v}  \sum\limits_{j=0}^{v-1} ||a_m||^{-2} \Big(\sum\limits_{\substack{2^jM^v<m\le 2^{j+1}M^v}}
 |c_m|^2\Big)^{1/v} C_2'(Q,2^jM^v)^{1/v}.
\end{equation}
Using the Cauchy-Schwarz inequality and the well-known bound $d_v(m) \ll  m^{\varepsilon}$ for the divisor function of order $v$, 
we obtain
\begin{equation*}
\sum\limits_{m} |c_m|^2
\ll M^{\varepsilon}
\sum\limits_{m}\
\sum\limits_{m_1\cdots m_v=m} |a_{m_1}\cdots a_{m_v}|^2 =
M^{\varepsilon} \left(
\sum\limits_{m} |a_m|^2\right)^v.
\end{equation*}
Combining this with \eqref{normbound} proves \eqref{C22}. $\Box$

\subsection{Proof of Lemma \ref{normlemma}, estimates \eqref{B11}-\eqref{B43}}
In this section, we establish the remaining estimates \eqref{B11}-\eqref{B43} in Lemma \ref{normlemma} in which the norms $B_i(Q,M)$ are compared. The estimate \eqref{B11} says that the norm $B_1(Q,M)$ is essentially increasing in $Q$.
It is easy to describe the idea behind the proof: simply take coefficients $a_m$ supported on multiples of a fixed prime $p$.  This extends the size of $Q$ by a factor $N(p)$ without essentially changing the size of the norm $B_1$.  There is a slight technical issue regarding coprimality with $p$ that can be circumvented by averaging over $p$.  The details are essentially the same as in Lemma 9 of \cite{Hea0} and we therefore omit this proof.

Next, we compare $B_1$ and $B_2$. We recall that in the definition \eqref{B1def} of $B_1$, the outer sum ranges over squarefree $n \in \mz[\omega]$ that are not divisible by any rational prime. We further recall that $B_2$ is defined in the same way as $B_1$ with the condition that $n$ is not divisible by any rational prime being removed. Hence, we have the trivial inequality $B_1(Q,M)\le B_2(Q,M)$.
Conversely, we want to prove the estimate \eqref{B21} of $B_2$ in terms of $B_1$.  To reduce the sum over squarefree $n\in \mz[\omega]$ to sums over squarefree $n\in \mz[\omega]$ that are not divisible by any rational prime, we extract rational divisors, getting 
\begin{equation*}
T^*(Q,M) := \sumstar\limits_{\substack{Q<N(n)\le 2Q\\ n\equiv 1 \shortmod 3}} \Big|
\sumstar\limits_{m} a_m \chi_m(n)\Big|^2
\leq \sum\limits_{\substack{|k| \le \sqrt{2Q}\\ k \equiv 1 \shortmod{3}}}\ \sumprime\limits_{\substack{Q/k^2<N(n)\le 2Q/k^2\\ n\equiv  1 \shortmod 3}} \Big|
\sumstar\limits_{m} a_m \chi_m(n)
\Big|^2,
\end{equation*}
Breaking the outer sum over $k$ on the right-hand side into $O(\log 2Q)$ dyadic intervals, we find that
\begin{multline*}
T^*(Q,M) \ll  \log(2Q) \sup\limits_{1\le X\le Q}\ \sumstar\limits_{(Q/X)^{1/2} \le k\le 2(Q/X)^{1/2}}\ \sumprime\limits_{\substack{X/4<N(n)\le 2X\\ n\equiv \pm 1 \shortmod 3}} \Big|
\sumstar\limits_{m} a_m \chi_m(n)
\Big|^2
\\
\ll (\log 2Q) \sup\limits_{1\le X\le Q} Q^{1/2}X^{-1/2} (B_1(X/4,M)+B_1(X/2,M)+B_1(X,M))||a_m||^2.
\end{multline*}
Combining this with the increasing property \eqref{B11} implies \eqref{B21}.
The proof of \eqref{B32} is similar to that of \eqref{B21} so we omit the details.

Finally, we compare $B_3$ and $B_4$. Since the proof of \eqref{B34} is essentially the same as that of Lemma 7 in \cite{Hea}, we omit it.  The idea is simply to extract the greatest common divisor of $m_1$ and $m_2$.
To derive the bound \eqref{B43} of $B_4$ in terms of $B_3$, we apply Lemma \ref{2dpoisson} to the sum corresponding to $B_4$, getting
\begin{multline*} \label{twodimpoisson}
\sumstar\limits_{(m_1,m_2)=1} a_{m_1}\overline{a_{m_2}}
\sum\limits_{n \in \mz[\omega]}  W\left(\frac{N(n)}{Q}\right) \left(\frac{n}{m_1}\right)_3\overline{\left(\frac{n}{m_2}\right)_3}
\\
=
  Q \sum\limits_{k\in \mz[\omega]}\ \sumstar\limits_{(m_1,m_2)=1} b_{m_1}\overline{b_{m_2}} \check{W}\left(\sqrt{\frac{N(k)Q}{(m_1m_2)^2}}\right) \overline{\left(\frac{k}{m_1}\right)_3}\left(\frac{k}{m_2}\right)_3,
\end{multline*}
with $\check{W}$ being a certain weight function of rapid decay and
\begin{equation*}
b_m:=a_m\leg{\sqrt{-3}}{m}_3 \frac{g(m)}{m^2}.
\end{equation*}
Now, similarly as in \cite{Hea}, we separate the variables $m_1$ and $m_2$ using
the Mellin transform of the weight function $\check{W}$, and using M\"{o}bius inversion on the coprimality condition $(m_1,m_2)=1$. We then use the Cauchy-Schwarz inequality, and after a short calculation, 
arrive at the estimate for $M \geq 1$
\begin{equation} \label{B430}
B_4(Q,M) \ll QM^{\varepsilon-2} \max\left\{B_3(K,M) : K\le M^4Q^{-1}\right\}
+ M^{6+\varepsilon}/Q \sum\limits_{K> M^4/Q} K^{-2-\varepsilon} B_3(K,M) 
\end{equation}
where $K$ runs over powers of $2$.  This corresponds to Lemma 8 in \cite{Hea}. We also have the trivial bound $B_4(Q,M) \ll Q$ if $M<1$.
Combining this with \eqref{B430}, we get \eqref{B43}. $\Box$


\begin{thebibliography}{99}
\bibitem[BL]{BL} Baier, S.; Zhao, L. {\it On the low-lying zeros of Hasse-Weil $L$-functions for elliptic curves}, 	arXiv:0708.2987v3 [math.NT].
\bibitem[BFH]{BFH} Brubaker, B.; Friedberg, S.; Hoffstein, J. {\it Cubic twists of ${\rm GL}(2)$ automorphic $L$-functions.}
Invent. Math. 160 (2005), no. 1, 31--58.
\bibitem[CFKRS]{CFKRS} Conrey, J.; Farmer, D.; Keating, J.; Rubinstein, M.; Snaith, N. {\it Integral moments of $L$-functions}, to appear in Proc. London Math. Soc.
\bibitem[Da]{Davenport} Davenport, H. {\it Multiplicative number theory.} Third edition. Revised and with a preface by Hugh L. Montgomery. Graduate Texts in Mathematics, 74. Springer-Verlag, New York, 2000.
\bibitem[DFK]{David} David, C.; Fearnley, J.; Kisilevsky, H. {\em On the vanishing of twisted $L$-functions of elliptic curves.}  Experiment. Math.  13  (2004),  no. 2, 185--198.
\bibitem[Di]{Diaconu} Diaconu, A. {\it Mean square values of Hecke $L$-series formed with $r$-th order characters.}  Invent. Math.  157  (2004),  no. 3, 635--684.
\bibitem[DT]{DiTi} Diaconu, A.; Tian, Y. {\it Twisted Fermat curves over totally real fields.} Ann. of Math. (2) 162 (2005), no. 3, 1353--1376.
\bibitem[E]{Elliott} Elliott, P. D. T. A. {\it On the mean value of $f(p)$.}  Proc. London Math. Soc. (3)  21  1970 28--96.
\bibitem[FaHL]{FHL} Farmer, D.; Hoffstein, J.; Lieman, D.
{\it Average values of cubic $L$-series.} Automorphic forms, automorphic representations, and arithmetic (Fort Worth, TX, 1996), 27--34, Proc. Sympos. Pure Math., 66, Part 2, Amer. Math. Soc., Providence, RI, 1999.
\bibitem[FKK]{FKK} Fearnley, J.; Kisileveky, H.; Kuwata, M. {\em Vanishing and Non-Vanishing Dirichlet Twists of L-Functions of Elliptic Curves.} arXiv:0711.1771v1 [math.NT].
\bibitem[FrHL]{FrHL} Friedberg, S.; Hoffstein, J.; Lieman, D. {\it Double Dirichlet series and the $n$-th order twists of Hecke $L$-series.}  Math. Ann.  327  (2003),  no. 2, 315--338.
\bibitem[Hea1]{Hea0} Heath-Brown, D.R. {\it A mean value estimate for real character sums.}
Acta Arith. 72 (1995), 235--275.
\bibitem[Hea2]{Hea} Heath-Brown, D.R. {\it Kummer's conjecture for cubic Gauss sums.}
Isr. J. Math. 120 (2000), 97--124.
\bibitem[H-BP]{H-BP} Heath-Brown, D. R.; Patterson, S. J.
	{\it The distribution of Kummer sums at prime arguments.}
	J. Reine Angew. Math. 310 (1979), 111--130.
\bibitem[IR]{IR} Ireland, K.; Rosen, M. {\it A classical introduction to modern number theory.} Second edition. Graduate Texts in Mathematics, 84. Springer-Verlag, New York, 1990.
\bibitem[I1]{I} Iwaniec, H. {\it Topics in Classical Automorphic Forms}, Graduate Studies in Mathematics, 17. American Mathematical Society, Providence, RI, 1997.
\bibitem[I2]{Iwaniec} Iwaniec, H. {\it On the order of vanishing of modular $L$-functions at the critical point.}  S\'{e}m. Th\'{e}or. Nombres Bordeaux (2)  2  (1990),  no. 2, 365--376.
\bibitem[IK]{IK} Iwaniec, H.; Kowalski, E. {\it Analytic Number Theory}. American Mathematical Society Colloquium Publications, 53.  American Mathematical Society, Providence, RI, 2004.
\bibitem[KS1]{KSM} Katz, N.; Sarnak, P. {\it Random Matrices, Frobenius Eigenvalues, and Monodromy}. American Mathematical Society Colloquium Publications, 45. American Mathematical Society, Providence, RI, 1999.
\bibitem[KS2]{KS} Katz, N.; Sarnak, P. {\it Zeroes of zeta functions and symmetry}, Bull. Amer. Math. Soc., {\bf 36}, 1-26 (1999).
\bibitem[Lem]{Lem} Lemmermeyer, F. {\it Reciprocity laws. From Euler to Eisenstein.}
Springer Monographs in Mathematics, Springer, Berlin (2000).
\bibitem[LP]{LP} Livn\'{e}, R.; Patterson, S. J. {\it The first moment of cubic exponential sums.}  Invent. Math.  148  (2002),  no. 1, 79--116.
\bibitem[L]{Luo} Luo, W. {\it On Hecke $L$-series associated with cubic characters.}
Compos. Math. 140 (2004), no. 5, 1191--1196.
\bibitem[P1]{P} Patterson, S. J. {\it A cubic analogue of the theta series.}  J. Reine Angew. Math.  296  (1977), 125--161.
\bibitem[P2]{P2} Patterson, S. J. {\it The distribution of general Gauss sums and similar arithmetic functions at prime arguments.}
Proc. London Math. Soc. (3) 54 (1987), no. 2, 193--215.
\bibitem[Y]{Y2} Young, M. P. {\it The first moment of quadratic Dirichlet $L$-functions}, Acta Arith. 138 (2009) no.1, 73--99.
\end{thebibliography}
\end{document}